\documentclass[a4paper]{amsart}

\usepackage[T1]{fontenc}
\usepackage{palatino}
\usepackage[svgnames]{xcolor}
\usepackage[pdftitle={Modest Sets are Equivalent to PERs}, pdfauthor={Rahul Chhabra}]{hyperref}
\hypersetup{
  colorlinks=true,
  urlcolor=blue,
  citecolor=teal,
  linkcolor=teal
}
\usepackage{orcidlink}
\usepackage{tikz}
\usepackage{tikz-cd}
\usepackage{amssymb}
\usepackage{amsthm}
\usepackage[toc,page,title]{appendix}
\usepackage{microtype}
\usepackage{fontawesome5}
\usepackage[marginratio=1:1,height=9in,width=6.5in]{geometry}

\newcommand{\Univ}[1]{\ensuremath{\mathcal{#1}}}

\newcommand{\Sets}[1]{\ensuremath{\mathsf{Sets}_{#1}}}
\newcommand{\SetsCat}[1]{\ensuremath{\mathsf{Sets}_{#1}}}

\newcommand{\Pair}[2]{\ensuremath{\langle #1, #2 \rangle}}

\newcommand{\Define}{\ensuremath{:\equiv}}

\newcommand{\PER}{\ensuremath{\mathsf{PER}_{\mathbb{A}}}}
\newcommand{\Asm}{\ensuremath{\mathsf{Asm}_{\mathbb{A}}}}
\newcommand{\Mod}{\ensuremath{\mathsf{Mod}_{\mathbb{A}}}}

\newcommand{\subQuot}[1]{\ensuremath{\mathsf{SubQuotient}(#1)}}
\newcommand{\canonicalPER}[1]{\ensuremath{\mathsf{CanonicalPER}(#1)}}

\newcommand{\PCA}{\ensuremath{\mathbb{A}}}
\newcommand{\Prop}[1]{\ensuremath{\mathsf{Prop}_{#1}}}
\newcommand{\Realizer}[1]{\ensuremath{\mathtt{#1}}}
\newcommand{\Reasoning}[1]{\{ #1 \}}
\newcommand{\id}{\ensuremath{\mathsf{id}}}
\newcommand{\Dom}[1]{\ensuremath{\mathsf{Dom}(#1)}}
\newcommand{\formalisation}[1]{\raisebox{-0.5pt}{\scalebox{0.8}{\href{#1}{\faCog}}}}
\newcommand{\fwd}{\ensuremath{\mathsf{forward}}}
\newcommand{\fwdMain}{\ensuremath{\mathsf{forwardMainMap}}}
\newcommand{\bwd}{\ensuremath{\mathsf{backward}}}
\newcommand{\bwdMain}{\ensuremath{\mathsf{backwardMainMap}}}
\newcommand{\propTrunc}[1]{\ensuremath{\parallel #1 \parallel_{-1}}}
\newcommand{\ptInc}[1]{\ensuremath{| #1 |_{-1}}}
\newcommand{\reasoning}[1]{\{ #1 \}}
\newcommand{\Triple}[3]{\ensuremath{\langle #1, #2, #3 \rangle}}
\newcommand{\perify}[1]{\ensuremath{\mathsf{perify}(#1)}}

\theoremstyle{definition}
\newtheorem{theorem}{Theorem}[section]
\newtheorem{lemma}[theorem]{Lemma}
\newtheorem{corollary}[theorem]{Corollary}
\newtheorem{intuition}[theorem]{Intuition}
\newtheorem{definition}[theorem]{Definition}

\newtheorem{remark}[theorem]{Remark}

\newtheorem{construction}[theorem]{Construction}

\newtheorem{notation}[theorem]{Notation}
\newtheorem{convention}[theorem]{Convention}

\usepackage[style=alphabetic]{biblatex}
\begin{document}
\title{Modest Sets are Equivalent to PERs}
\author{
  Rahul Chhabra \orcidlink{0009-0007-7917-6461}
}
\date{\today}

\begin{abstract}
  The aim of this article is to give an expository account of the equivalence between modest sets and partial equivalence relations. Our proof is entirely self-contained in that we do not assume any knowledge of categorical realizability. At the heart of the equivalence lies the \emph{subquotient construction} on a partial equivalence relation. The subquotient construction embeds the category of partial equivalence relations into the category of modest sets. We show that this embedding is a split essentially surjective functor, and thereby, an equivalence of categories. Our development is both constructive and predicative, and employs the language of \emph{homotopy type theory}. All the mathematics presented in this article has been mechanised in Cubical Agda.
\end{abstract}
\maketitle

\section{Introduction}
\label{sec:introduction}

Kleene, in his seminal work \cite{kleene-number-realizability} developed the technique of ``recursive realizability'' to show how to give a \emph{computational} meaning to the logic of Brouwer's intuitionism. The realizability relationship is written $\Realizer{a} \Vdash x$ and is read ``\Realizer{a} realizes $x$'' or ``\Realizer{a} is a realizer of $x$''. We want to think of the realizability relationship as capturing the intuition that \Realizer{a} is some kind of ``machine code representation'' of a mathematical object $x$.

The discipline of categorical realizability emerged with Martin Hyland's discovery of the effective topos \cite{hyland-effective-topos}. Categorical realizability studies \emph{realizability toposes} (such as the effective topos) and important subcategories thereof through the lens of categorical logic and topos theory.

The effective topos is the ``world of recursive realizability''. It massively generalises Kleene's original realizability interpretation from Heyting arithmetic to topos logic. Additionally, it serves as a semantic playground for studying computability theory \cite{bauer-synthetic-computability-theory} and domain theory ``synthetically'' \cite{hyland-synthetic-domain-theory}.


Recall that a theorem due to Freyd \cite{freyd-abelian-categories} tells us that any small complete category in \SetsCat{} is necessarily posetal. This result extends to all Grothendieck toposes, so that, in any Grothendieck topos, a small complete category is necessarily posetal \cite{shulman-small-complete-category-grothendieck} \cite{gubkin-freyds-theorem-grothendieck-topos}. The category of partial equivalence relations (when arranged as an internal category in the effective topos), however, \emph{is} a non-posetal small complete category \cite{hyland-small-complete-category}. This fact is notable since it clearly highlights how realizability toposes differ from Grothendieck toposes. That this category is small and complete is also useful in the semantics of the polymorphic $\lambda$-calculus \cite{longo-moggi-omega-set}.


Both modest sets and partial equivalence relations capture the intuitive notion of a ``data type''. It is fitting that they are equivalent.
The aim of this article is to give a detailed and self-contained proof of the equivalence. Obviously, this fact is well-known to experts. It already shows up in Hyland's articles on the effective topos \cite[Proposition 7.2]{hyland-effective-topos} and on PERs \cite[p.~151]{hyland-small-complete-category}. I claim no originality for the mathematics presented here. Mistakes and opinions are my own.

Our development is completely \textbf{constructive} in that we do not assume the Law of Excluded Middle (LEM) \cite[Section 3.4]{hott-book} or the Axiom of Choice (AC) \cite[Section 3.8]{hott-book}. It is also \textbf{predicative} in that we do not assume the axiom of propositional resizing \cite[Axiom 3.5.5]{hott-book}. That said, experience suggests that any interesting applications of this equivalence require us to assume propositional resizing.

\subsection*{Acknowledgments}
I have benefited greatly from many insightful discussions with Jon Sterling. Additionally, Jon explained to me the main idea behind the construction of the backwards direction of the isomorphism in Section~\ref{sec:backwards-direction}. Tom de Jong and Ian Ray gave me feedback on an earlier draft of this article. Tom also taught me a LaTeX hack to get the \formalisation{} symbol to behave well. I am grateful to them for the support and time they have given me.

\section{Univalent Foundations}
\label{sec:univalent-foundations}

The aim of this section is to establish notation and conventions, not to give a comprehensive review of univalent foundations. Readers unfamiliar with HoTT/UF should consult the HoTT Book \cite{hott-book}. We will work in homotopy type theory \cite{hott-book}, an extension of Martin-L{\"o}f's intensional type theory with the \textbf{univalence axiom} and \textbf{higher inductive types}. More explicitly, we work in Martin-L{\"o}f's intensional type theory with the empty type $\mathbf{0}$, the unit type $\mathbf{1}$, dependent products $(a : A) \times B(a)$, dependent functions $(a : A) \to B(a)$, identity/path types $a =_A b$, inductive types, and a hierarchy of universes $\Univ{U}_0 : \Univ{U}_1 : \Univ{U}_2 : \dots$. When the type $A$ in a path type $a =_A b$ is obvious, we omit it. We use the symbol $\doteq$ for definitional or judgemental equality. Note that we prefer to use the Nuprl/Agda notation for dependent products $(a : A) \times B(a)$ and dependent functions $(a : A) \to B(a)$, as opposed to the HoTT Book's $\sum_{a : A} B(a)$ and $\prod_{a : A} B(a)$ respectively.

We assume familiarity with the concept of $n$-types. We call $-1$-types the \textbf{propositions} and $0$-types the \textbf{sets}. The identity types of sets have no structure; as such, when $A$ is a set and we want to talk about the identity type of $A$, we will generally not bother with witnesses and simply say that two elements are ``equal''. 
For a universe \Univ{U}, we write $\Prop{\Univ{U}}$ for the collection of propositions in \Univ{U}, and $\Sets{\Univ{U}}$ for the collection of sets in \Univ{U}.

We use only two higher inductive types --- propositional truncation and set quotient. We write $A / R$ for the set quotient of the type $A$ by the relation $R : A \to A \to \Univ{U}$ and $[\_] : A \to A / R$ for the surjective inclusion of $A$ into $A / R$. We write $\propTrunc{A}$ for the propositional truncation (i.e., $-1$-truncation) of $A$ and $\ptInc{\_} : A \to \propTrunc{A}$ for the inclusion of $A$ into $\propTrunc{A}$. We assume working knowledge of the elimination principles of propositional truncation and the set quotient.

I wish to recall how we can eliminate propositional truncations into sets. It can be shown that functions $\propTrunc{A} \to B$ are exactly the coherently constant functions $A \xrightarrow{\omega} B$ \cite{kraus-coherently-constant}. When $B$ is a set, the coherently constant functions ${A} \xrightarrow{\omega} B$ are exactly the functions that are na\"ively constant, that is, functions $f : A \to B$ such that for any $a_1 , a_2 : A$, $f(a_1)$ equals $f(a_2)$. This gives us a simple way to define functions from a proposition truncation $\propTrunc{A}$ into a set $B$.

For the most part, we use the conventions of the HoTT Book \cite{hott-book}. We differ from the HoTT Book in our usage of the words ``exist'' and ``category''. We use the word ``existence'' where the HoTT Book would use ``mere existence'', and we use the word ``category'' where the HoTT Book uses ``precategory''.

We shall assume that the reader is familiar with the ideas of Part I of the HoTT Book \cite[Part I]{hott-book} and with the HoTT Book's treatment of category theory \cite[Chapter 9]{hott-book}.

\subsection{Agda Formalisation}

One of the advantages of univalent type theory is that the verification of large amounts of mathematics becomes practical. Cubical Agda \cite{cubical-agda} is a proof assistant based on the cubical sets interpretation of homotopy type theory \cite{cchm}. \sloppy All the mathematics presented in this article has been verified in Cubical Agda. Since our development is completely constructive and predicative, Cubical Agda can actually run our proofs as computer programs! This is an advantage of Cubical Agda over vanilla Agda with univalence assumed as an axiom.

 A Git repository of the Agda code is available at \url{https://github.com/rahulc29/realizability/}. An HTML rendering can also be found at \url{https://rahulc29.github.io/realizability/}. Throughout the article, certain definitions, theorems and lemmas are accompanied by a \formalisation{} icon that links to the corresponding item in the formal development. 

\section{Combinatory Algebras, Assemblies, and Modest Sets}

\subsection{Combinatory Algebras}
\label{sec:combinatory-algebras}

While Kleene's recursive realizability uses natural numbers, there is a large class of structures that give a notion of realizer. The one we focus on is the class of \emph{combinatory algebras}.


\begin{definition}[\formalisation{https://rahulc29.github.io/realizability/Realizability.ApplicativeStructure.html\#1414}]
  An \textbf{applicative structure} on a set \PCA~ consists of an ``application'' function $\mathsf{appl} : \PCA \to \PCA \to \PCA$. We use juxtaposition to denote the application function, that is, we write \Realizer{a \; b} for $\mathsf{appl}(\Realizer{a}, \Realizer{b})$.
\end{definition}

Notice that our definition of applicative structure calls for a \emph{total} application function. In much of categorical realizability, one works with \emph{partial} applicative structures where the application function is partial. We ignore issues of partiality here by sticking to total applicative structures and total combinatory algebras, while mentioning that our results hold for the partial case as well.

\begin{remark}
  It is not particularly difficult to deal with partiality using the standard techniques for talking about partiality in univalent mathematics \cite{knapp-partiality}. However, in Agda, it does lead to a large number of bureaucratic proof obligations where we have to justify that the realizer we want to talk about is defined. For this reason, we have favoured total combinatory algebras.
\end{remark}

\begin{definition}[\formalisation{https://rahulc29.github.io/realizability/Realizability.CombinatoryAlgebra.html\#296}]
  A (total) \textbf{combinatory algebra} structure on an applicative structure $\Pair{\PCA}{\mathsf{appl}}$ consists of a choice of combinators \Realizer{s} and \Realizer{k} such that
  \begin{itemize}
  \item For any $\Realizer{a}, \Realizer{b}, \Realizer{c} : \PCA$, we have that $\Realizer{s \; a \; b \; c} = \Realizer{a \; c \; (b \; c)}$ holds.
    \item For any $\Realizer{a}, \Realizer{b} : \PCA$, we have that $\Realizer{k \; a \; b} = \Realizer{a}$ holds.
  \end{itemize}
\end{definition}

The definition of a combinatory algebra is a deceptively simple definition, and it is not obvious how such a structure captures our intuitive notion of realizer. However, we can derive from the combinators \Realizer{s} and \Realizer{k} the following features :
\begin{itemize}
\item An identity combinator --- a combinator \Realizer{i} such that for any \Realizer{a} : \PCA~ we have that $\Realizer{i \; a} = \Realizer{a}$ holds.
\item A pairing combinator \Realizer{pair} along with projection combinators \Realizer{pr1} and \Realizer{pr2} such that for any \Realizer{a} and \Realizer{b} in \PCA~ we have that $\Realizer{pr1 \; (pair \; a \; b)} = \Realizer{a}$ and $\Realizer{pr2 \; (pair \; a \; b)} = \Realizer{b}$ hold.
\item An encoding of the Booleans \Realizer{true} and \Realizer{false} along with an \Realizer{ifThenElse} combinator such that for any \Realizer{t} and \Realizer{e} in \PCA~ we have that $\Realizer{ifThenElse \; true \; t \; e} = \Realizer{t}$ and $\Realizer{ifThenElse \; false \; t \; e} = \Realizer{e}$ hold.
  \item A fixedpoint combinator \Realizer{Y} such that for any \Realizer{f} : \PCA~ we have that $\Realizer{f \; (Y \; f)} = \Realizer{Y \; f}$ holds.
  \item An encoding of the natural numbers known as the \textbf{Curry numerals} along with a primitive recursion combinator \Realizer{primRec}.
  \item An encoding of the partial recursive functions.
  \item An analogue of $\lambda$-abstraction and $\beta$-reduction\footnote{In the presence of partiality, $\beta$-reduction is not well-behaved \cite[Theorem 1.1.9]{longley-phd-thesis}.} that we write $\mathtt{\backslash \; x \to f}$.
\end{itemize}

It is these combinators that allow us to think of a combinatory algebra as a simple, untyped (low-level) functional programming language. For this reason, we use a $\mathtt{monospace}$ font for elements of \PCA. The actual derivations of these combinators may be found in, say, van Oosten's textbook \cite[Section 1.1.1, Section 1.3.1]{van-oosten-textbook}.

\subsection{Assemblies and Modest Sets}

From this point onwards, we will assume that we have a universe \Univ{U} in which \PCA~ lives. Unless stated otherwise, whenever we speak of a set or a proposition, we always mean a \Univ{U}-small set and a \Univ{U}-small proposition respectively.

\begin{definition}[\formalisation{https://rahulc29.github.io/realizability/Realizability.Assembly.Base.html\#580}]
  An \textbf{assembly} on a carrier set $X$ consists of a \textbf{realizability relation} $\_\Vdash_X\_ : \PCA \to X \to \Prop{}$ such that for any $x : X$, there exists a realizer. That is, it is a triple $\Triple{X}{\_\Vdash_X\_}{e_X}$ where $X$ is a set, $\_\Vdash_X\_$ is the realizability relation and, for any $x : X$, $e_X(x)$ is evidence that $x$ is realized.
\end{definition}

\begin{notation}
  When the assembly is obvious, we will write $e_{\Realizer{r}}[x]$ for evidence of the proposition $\Realizer{r} \Vdash x$.
\end{notation}

When $\Realizer{a} \Vdash_X x$ holds we say that ``\Realizer{a} realizes $x$'' or ``\Realizer{a} is a realizer of $x$''. Intuitively, an assembly is just a set equipped with an ``encoding'' of it's elements in \PCA. Notice that we can present the data of an assembly on a carrier $A$ as a function $E_A : A \to \mathcal{P}^+(\PCA)$ where $\mathcal{P}^+(\PCA)$ is the ``inhabited powerset'' of \PCA, the set of all subsets of \PCA~ that are inhabited. For any $a : A$, $E_A(a)$ is the set of all realizers of $a$. 

Assemblies form a category with an obvious notion of map between assemblies $\Triple{A}{\_\Vdash_A\_}{e_A}$ and $\Triple{B}{\_\Vdash_B\_}{e_B}$ --- it is a map $f : A \to B$ for which there exists an element \Realizer{t} in \PCA~ that can compute a correct encoding for $f(x)$ given an encoding of $x$. Such a \Realizer{t} is known as a tracker for $f$.

\begin{definition}[\formalisation{https://rahulc29.github.io/realizability/Realizability.Assembly.Morphism.html\#943}]
  For assemblies $\Triple{A}{\_\Vdash_A\_}{e_A}$ and $\Triple{B}{\_\Vdash_B\_}{e_B}$, and a map $f : A \to B$, a \textbf{tracker} for $f$ is an element \Realizer{t} such that whenever \Realizer{a} realizes $a : A$, we have that \Realizer{t \; a} realizes $f(a) : B$.
\end{definition}

\begin{definition}[\formalisation{https://rahulc29.github.io/realizability/Realizability.Assembly.Morphism.html\#1292}]
  For assemblies $\Triple{A}{\_\Vdash_A\_}{e_A}$ and $\Triple{B}{\_\Vdash_B\_}{e_B}$, an \textbf{assembly morphism} is a map $f : A \to B$ such that there exists a tracker for $f$. That is, an assembly morphism between $\Triple{A}{\_\Vdash_A\_}{e_A}$ and $\Triple{B}{\_\Vdash_B\_}{e_B}$ is a pair \Pair{f}{e_f} where $f : A \to B$ is a map and $e_f$ is evidence that a tracker for $f$ exists.
\end{definition}

The following observation is trivial but useful:

\begin{lemma}[Assembly Morphism Extensionality, \formalisation{https://rahulc29.github.io/realizability/Realizability.Assembly.Morphism.html\#2435}]\label{lemma:assembly-morphism-extensionality}
  For assemblies $\Triple{A}{\_\Vdash_A\_}{e_A}$ and $\Triple{B}{\_\Vdash_B\_}{e_B}$ and assembly morphisms $\Pair{f}{e_f} , \Pair{g}{e_g} : \Triple{A}{\_\Vdash_A\_}{e_A} \to \Triple{B}{\_\Vdash_B\_}{e_B}$, the identity type of assembly morphisms $\Pair{f}{e_f} = \Pair{g}{e_g}$ is equivalent to the identity type of the underlying functions $f = g$, which is further equivalent to the type $(x : A) \to f(x) = g(x)$.
\end{lemma}

\begin{lemma}[\formalisation{https://rahulc29.github.io/realizability/Realizability.Assembly.Morphism.html\#5604}]
  Assemblies equipped with assembly morphisms and the obvious notion of identity and composition form a category \Asm.
\end{lemma}

Among other things, the category \Asm~ has finite limits, finite colimits, is (locally) cartesian closed, regular (but not exact), and has a classifier for regular subobjects. It is ``nearly a topos'' --- it is a \emph{quasitopos}. Additionally, it is equivalent to the category of double-negation separated objects of the realizability topos. 

\begin{notation}
In this article, we rarely deal with different assemblies on the same carrier. As such, we shall abuse notation and write ``let $A$ be an assembly'' to mean ``let $\Triple{A}{\_\Vdash_A\_}{e_A}$ be an assembly on $A$''. Similarly, we shall write ``let $f : A \to B$ be an assembly morphism'' to mean ``let $\Pair{f}{e_f}$ be an assembly morphism from $\Triple{A}{\_\Vdash_A\_}{e_A}$ to $\Triple{B}{\_\Vdash_B\_}{e_B}$''.
\end{notation}

Since the realizability relation can be more or less arbitrary, it is possible for elements to \emph{share} realizers. Indeed, we can take any set $X$ and define a realizability relation on $X$ by saying that every $\Realizer{a} : \PCA$ realizes every $x : X$. This gives us the \textbf{codiscrete}\footnote{To elaborate, we have a functor $\nabla : \Sets{} \to \Asm$ that is right adjoint to the global sections functor, thus the name ``codiscrete''.} assembly on $X$.

This can be undesirable, partly because it violates the ``programmer's intuition'' of a data type. In a data type, one realizer should realize (at most) one element. To remedy this, we can impose the condition of modesty on an assembly:

\begin{definition}[\formalisation{https://rahulc29.github.io/realizability/Realizability.Modest.Base.html\#1057}]\label{def:modesty}
  An assembly $A$ is said to be \textbf{modest} if, for any elements $x : A$ and $y : A$, whenever there exists a realizer $\Realizer{a} : \PCA$ that realizes both $x$ and $y$, $x$ and $y$ are equal. More formally,
  \begin{align*}
    &\mathsf{isModest} : \Asm \to \Prop{} \\
    &\mathsf{isModest} \; \Triple{A}{\_\Vdash_A\_}{e_A} \; \Define \forall \; (x \; y : A) \to \exists (\Realizer{a} : \PCA) \; \Realizer{a} \Vdash_A x \times \Realizer{a} \Vdash_A y \to x =_A y
  \end{align*}
\end{definition}

In other words, elements do not share realizers. Notice that since the carrier of an assembly is a set, being modest is a property and not structure. Modest assemblies are also known as modest sets, and that is also what we will prefer to call them. Modest sets appeared in Hyland's original paper on the effective topos \cite{hyland-effective-topos} under the name ``strictly effective objects''. The terminological change to ``modest sets'' was suggested by Dana Scott.

\begin{remark}
  In the literature, one sometimes find a definition of modesty that states that ``for any $x : A$ and $y : A$ if $E_A(a) \cap E_A(b) \neq \emptyset$ then $a$ equals $b$''. This is a \emph{negative} formulation of the same statement we have given in [Definition~\ref{def:modesty}]. The two definitions are classically equivalent, but, in a constructive setting, our \emph{positive} formulation is easier to work with.
\end{remark}

Modest sets form a full subcategory of the category of assemblies that we shall denote as \Mod. Despite being smaller than \Asm, it has all finite limits and colimits, is (locally) cartesian closed and regular. Additionally, it is an \textbf{exponential ideal} in \Asm, so that if $Y$ is modest and $X$ is an assembly then the exponential object $Y^X$ is also modest.

\section{Partial Equivalence Relations}
\label{sec:partial-equivalence-relations}

In this section, we look at the category of partial equivalence relations in some detail.

\begin{definition}[\PER, \formalisation{https://rahulc29.github.io/realizability/Realizability.PERs.PER.html\#1618}]\label{def:per}
  To give a \textbf{partial equivalence relation} is to give a function $R : \PCA \to \PCA \to \Prop{\Univ{U}}$ along with
  \begin{enumerate}
  \item Evidence of \textbf{symmetry}, that is, evidence that for any $\Realizer{a} : \PCA$ and $\Realizer{b} : \PCA$, if $R(\Realizer{a},\Realizer{b})$ holds then $R(\Realizer{b},\Realizer{a})$ holds.
  \item Evidence of \textbf{transitivity}, that is, evidence that for any $\Realizer{a} \; \Realizer{b} \; \Realizer{c} : \PCA$ if $R(\Realizer{a},\Realizer{b})$ and $R(\Realizer{b},\Realizer{c})$ hold then $R(\Realizer{a},\Realizer{c})$ holds.
  \end{enumerate}
\end{definition}

\begin{convention}
  We shall use the word ``per'' to refer to a term of type \PER~ [Definition~\ref{def:per}].
\end{convention}

\begin{notation}
  We write $\Realizer{a} \approx_R \Realizer{b}$ for $R(\Realizer{a}, \Realizer{b})$. For instance, we shall write ``$\Realizer{i \; a} \approx_R \Realizer{i \; a}$ holds iff $\Realizer{a} \approx_R \Realizer{a}$ holds'' instead of ``$R(\Realizer{i \; a}, \Realizer{i \; a})$ holds iff $R(\Realizer{a}, \Realizer{a})$ holds''.
\end{notation}

\begin{definition}[Domain of a per, \formalisation{https://rahulc29.github.io/realizability/Realizability.PERs.SubQuotient.html\#1706}]\label{def:domain-per}
  The \textbf{domain} of a per $R$ is the type of all $\Realizer{a} :\PCA$ such that $\Realizer{a} \approx_R \Realizer{a}$ holds. For a per $R$, we write this type as \Dom{R}. 
\end{definition}

\begin{notation}
  We generally write elements of \Dom{R} as pairs \Pair{\Realizer{r}}{r_{\Realizer{r}}} where \Realizer{r} : \PCA~ is a realizer and $r_{\Realizer{r}} : \Realizer{r} \approx_R \Realizer{r}$ witnesses that \Realizer{r} is related to itself.
\end{notation}

\begin{intuition}
  We want to think about a per $R$ as encoding a \textbf{data type}. Realizers \Realizer{a} and \Realizer{b} that are in the domain of $R$ should be thought of as representing terms of the data type, and they represent the same term exactly when $\Realizer{a} \approx_R \Realizer{b}$ holds.
\end{intuition}

\subsection{Morphisms of \PER}
\label{sec:morphisms-of-pers}

Our terminology surrounding morphisms in \PER~ is taken from Phoa's notes \cite{phoa-lecture-notes}.

\begin{definition}[\formalisation{https://rahulc29.github.io/realizability/Realizability.PERs.PER.html\#3708}]\label{def:tracker}
  For pers $R$ and $S$ a \textbf{tracker} is a realizer $\Realizer{t} : \PCA$ that maps ``related elements to related elements''. Concretely, for any realizers $\Realizer{a} : \PCA$ and $\Realizer{b} : \PCA$ for which $\Realizer{a} \approx_R \Realizer{b}$ holds, we have that $\Realizer{t \; a} \approx_S \Realizer{t \; b}$ holds.
\end{definition}

\begin{intuition}
  Notwithstanding we have not yet defined a morphism of pers, a tracker \Realizer{t} between pers $R$ and $S$ should be thought of as a concrete \textbf{representation} of an actual map between $R$ and $S$.
\end{intuition}

\begin{definition}[\formalisation{https://rahulc29.github.io/realizability/Realizability.PERs.PER.html\#3890}]\label{def:equivalence-of-trackers}
  Two trackers \Realizer{t} and \Realizer{u} between pers $R$ and $S$ are \textbf{equivalent} if they map every realizer in the domain of $R$ to related elements in $S$. In other words, for any \Realizer{r} such that $\Realizer{r} \approx_R \Realizer{r}$ holds, we have that $\Realizer{t \; r} \approx_S \Realizer{u \; r}$ holds.
\end{definition}

\begin{lemma}[\formalisation{https://rahulc29.github.io/realizability/Realizability.PERs.PER.html\#4044}]
  Equivalence of trackers [Definition~\ref{def:equivalence-of-trackers}] is an equivalence relation.
\end{lemma}
\begin{proof}
  Reflexivity follows immediately by unfolding the definitions of tracker and equivalence of trackers. Symmetry and transitivity follow from the symmetry and transitivity of $S$.
\end{proof}

\begin{definition}[\formalisation{https://rahulc29.github.io/realizability/Realizability.PERs.PER.html\#4954}]\label{def:per-morphism}
  A \textbf{morphism of pers} from $R$ to $S$ is an equivalence class of trackers from $R$ to $S$ under equivalence of trackers.
\end{definition}

We can find identity and composites for these morphisms.

\begin{definition}[Identity Tracker, \formalisation{https://rahulc29.github.io/realizability/Realizability.PERs.PER.html\#5321}]\label{def:identity-tracker}
  The \textbf{identity tracker} of a per $R$ is the identity combinator \Realizer{i} of the combinatory algebra \PCA~ interpreted as a tracker.
\end{definition}
\begin{proof}
  Consider \Realizer{a} and \Realizer{b} in \PCA~ such that $\Realizer{a} \approx_R \Realizer{b}$ holds. We need to show $\Realizer{i \; a} \approx_R \Realizer{i \; b}$ holds. This follows immediately since $\Realizer{i \; a} = \Realizer{a}$ and $\Realizer{i \; b} = \Realizer{b}$ hold.
\end{proof}

\begin{definition}[Composite Tracker, \formalisation{https://rahulc29.github.io/realizability/Realizability.PERs.PER.html\#5477}]\label{def:composite-tracker}
  For trackers \Realizer{t} between pers $R$ and $S$ and \Realizer{u} between $S$ and $T$ the \textbf{composite} of \Realizer{t} and \Realizer{u} is given by \Realizer{\backslash x \to u \; (t \; x)}.
\end{definition}
\begin{proof}
  \sloppy Consider \Realizer{a} and \Realizer{b} in \PCA~ such that $\Realizer{a} \approx_R \Realizer{b}$ holds. We need to show that $\Realizer{(\backslash x \to u \; (t \; x)) \; a} \approx_T \Realizer{(\backslash x \to u \; (t \; x)) \; b}$ holds.

  A simple calculation
  \begin{align*}
    &\Realizer{a} \approx_R \Realizer{b} \\
    &\Reasoning{\Realizer{t}\text{ is a tracker}} \\
    &\Realizer{t \; a} \approx_S \Realizer{t \; b} \\
    &\Reasoning{\Realizer{u}\text{ is a tracker}} \\
    &\Realizer{u \; (t \; a)} \approx_T \Realizer{u \; (t \; b)} \\
    &\Reasoning{\lambda \text{-abstraction}} \\
    &\Realizer{(\backslash x \to u \; (t \; x)) \; a} \approx_T \Realizer{(\backslash x \to u \; (t \; x)) \; b}
  \end{align*}
  suffices to show this.
\end{proof}

\begin{construction}[Identity Morphism, \formalisation{https://rahulc29.github.io/realizability/Realizability.PERs.PER.html\#5321}]\label{constr:identity-morphism-pers}
  The \textbf{identity morphism} for a per $R$ is given by the equivalence class of the identity tracker [Definition~\ref{def:identity-tracker}].
\end{construction}

\begin{construction}[Composite Morphisms, \formalisation{https://rahulc29.github.io/realizability/Realizability.PERs.PER.html\#5871}]\label{constr:composite-morphism-pers}
  For pers $R$, $S$ and $T$ and for morphisms $f : R \to S$ and $g : S \to T$ we can construct the \textbf{composite morphism} $g \circ f : R \to T$.
\end{construction}
\begin{proof}
  We shall apply the set-level elimination principle of set-quotients on $f$ and $g$.

  To do this, we must first show how to define a composite morphism by assuming we have representatives for $f$ and $g$ and then show the coherence conditions --- that our definition is independent of the choice of representatives.

  To that end, let us assume we have been given representatives \Realizer{f} for $f$ and \Realizer{g} for $g$ respectively.

  We define the composite morphism as the equivalence class spanned by the composite tracker of \Realizer{f} and \Realizer{g} [Definition \ref{def:composite-tracker}].

  It remains to show that the coherence conditions hold.

  Let \Realizer{f} and \Realizer{f'} both be equivalent trackers of $f$. We will show that the equivalence classes $[\Realizer{\backslash x \to g \; (f \; x)}]$ and $[\Realizer{\backslash x \to g \; (f' \; x)}]$ are equal.

  To show that $[\Realizer{\backslash x \to g \; (f \; x)}]$ and $[\Realizer{\backslash x \to g \; (f' \; x)}]$ are equal, it suffices to show that the representatives are equivalent, that is, $\Realizer{\backslash x \to g \; (f \; x)}$ is equivalent as a tracker to $\Realizer{\backslash x \to g \; (f' \; x)}$.

  Let \Realizer{a} be an element of \PCA~ such that $\Realizer{a} \approx_R \Realizer{a}$ and calculate
  \begin{align*}
    &\Realizer{a} \approx_R \Realizer{a} \\
    &\Reasoning{\Realizer{f}\text{ is equivalent to }\Realizer{f'}} \\
    &\Realizer{f \; a} \approx_S \Realizer{f' \; a} \\
    &\Reasoning{\Realizer{g}\text{ is a tracker}} \\
    &\Realizer{g \; (f \; a)} \approx_T \Realizer{g \; (f' \; a)} \\
    &\Reasoning{\lambda \text{ abstraction in }\PCA} \\
    &\Realizer{\backslash x \to g \; (f \; x) \; a} \approx_T \Realizer{\backslash x \to g \; (f' \; x) \; a}
  \end{align*}
  to see that the equivalence indeed holds.

  A similar calculation shows that our definition is independent of the choice of representative for $g$.
\end{proof}

\subsection{The Category \PER}
\label{sec:category-per}

We need a few auxiliary lemmas before we can formally define the category of pers.

\begin{lemma}[Left Identity Law, \formalisation{https://rahulc29.github.io/realizability/Realizability.PERs.PER.html\#6724}]\label{lemma:left-identity-law-pers}
  The identity morphism [Construction~\ref{constr:identity-morphism-pers}] satisfies the \textbf{left identity law}, that is, for any pers $R$ and $S$ and per morphism $f : R \to S$ we have that $\id_S \circ f = f$ holds.
\end{lemma}

\begin{lemma}[Right Identity Law, \formalisation{https://rahulc29.github.io/realizability/Realizability.PERs.PER.html\#7135}]\label{lemma:right-identity-law-pers}
  The identity morphism [Construction~\ref{constr:identity-morphism-pers}] satisfies the \textbf{right identity law}, that is, for any pers $R$ and $S$ and per morphism $f : R \to S$ we have that $f \circ \id_R = f$ holds.
\end{lemma}

\begin{lemma}[Associativity Law, \formalisation{https://rahulc29.github.io/realizability/Realizability.PERs.PER.html\#7505}]\label{lemma:associativity-law-pers}
  Composition of morphisms [Construction~\ref{constr:composite-morphism-pers}] satisfies the \textbf{associativity law}, that is, for any pers $R$, $S$, $T$ and $U$ and for morphisms $f : R \to S$, $g : S \to T$, and $h : T \to U$ we have that $h \circ (g \circ f) = (h \circ g) \circ f$.
\end{lemma}

\begin{definition}[The Category of Partial Equivalence Relations, \formalisation{https://rahulc29.github.io/realizability/Realizability.PERs.PER.html\#8410}]\label{def:category-per}
  The category of partial equivalence relations is given by
  \begin{enumerate}
  \item The type of pers \PER [Definition~\ref{def:per}] as the type of objects
  \item The set of per morphisms [Definition~\ref{def:per-morphism}] as the set of morphisms
  \item The identity per morphism [Construction~\ref{constr:identity-morphism-pers}] as the identity morphism
  \item The composite of per morphisms [Construction~\ref{constr:composite-morphism-pers}] as the composite of morphisms
  \item The proofs of [Lemma~\ref{lemma:left-identity-law-pers}] and [Lemma~\ref{lemma:right-identity-law-pers}] as evidence for the identity laws
  \item The proof of [Lemma~\ref{lemma:associativity-law-pers}] as evidence for the associativity laws
  \end{enumerate}
\end{definition}

\section{The Subquotient Functor}
\label{sec:subquotient-functor}

In this section we shall define a functor $\subQuot{\_} : \PER \to \Mod$ that takes a per $R$ to the modest \textbf{subquotient assembly} \subQuot{R}.

\subsection{Action on Objects}
\label{sec:subquotient-functor-action-on-objects}

\begin{definition}[Subquotient of a per, \formalisation{https://rahulc29.github.io/realizability/Realizability.PERs.SubQuotient.html\#1771}]\label{def:subquotient-set-per}
  For a per $R$ we define the \textbf{subquotient set} of $R$ as the set obtained by quotienting the domain of $R$ [Definition~\ref{def:domain-per}] with the relation underlying $R$.
\end{definition}
\begin{convention}[Terms of the Subquotient]\label{conv:term-subquotient}
  For a per $R$ we shall refer to terms of the subquotient \subQuot{R} as \textbf{partial equivalence classes}.
\end{convention}


The subquotient set of a per [Definition~\ref{def:subquotient-set-per}]  will serve as the carrier of the subquotient assembly. To define the realizability relation for the subquotient assembly we will have to perform a ``large elimination'' --- we will have to eliminate a partial equivalence class into the universe.

Since the subquotient of a per $R$ is constructed using a higher inductive type, we will proceed in two steps for this large elimination.
\begin{enumerate}
\item Define a \emph{pre-realizability relation} $\mathsf{preReal} : \PCA \to \Dom{R} \to \Prop{}$
\item Show the \emph{coherence condition} --- for \Pair{\Realizer{a}}{r_{\Realizer{a}}} and \Pair{\Realizer{b}}{r_{\Realizer{b}}} in \Dom{R} with $\Realizer{a} \approx_R \Realizer{b}$ we have that $\mathsf{preReal}(\Pair{\Realizer{a}}{r_{\Realizer{a}}}) = \mathsf{preReal}({\Pair{\Realizer{b}}{r_{\Realizer{b}}}})$ holds.
\end{enumerate}

\begin{definition}[Pre-realizability Relation]\label{def:pre-realizability}
  We say that $\Realizer{a} : \PCA$ pre-realizes \Pair{\Realizer{b}}{r_{\Realizer{b}}} if $\Realizer{a} \approx_R \Realizer{b}$ holds.
\end{definition}

\begin{lemma}[Coherence for Pre-Realizability]\label{lemma:coherence-for-pre-realizability}
  \sloppy For \Pair{\Realizer{a}}{r_{\Realizer{a}}} and \Pair{\Realizer{b}}{r_{\Realizer{b}}} in \Dom{R} with $\Realizer{a} \approx_R \Realizer{b}$ we have that $\mathsf{preReal}(\Pair{\Realizer{a}}{r_{\Realizer{a}}}) = \mathsf{preReal}(\Pair{\Realizer{b}}{r_{\Realizer{b}}})$ holds.
\end{lemma}
\begin{proof}
  By propositional and function extensionality, it suffices to show that for any $\Realizer{r} : \PCA$, the propositions $\mathsf{preReal}(\Pair{\Realizer{a}}{r_{\Realizer{a}}})(\Realizer{r})$ and $\mathsf{preReal}(\Pair{\Realizer{b}}{r_{\Realizer{b}}})(\Realizer{r})$ imply each other. Unfolding definitions, we need to show that $\Realizer{r} \approx_R \Realizer{a}$ holds iff $\Realizer{r} \approx_R \Realizer{b}$ holds. This follows easily from the symmetry and transitivity of $R$.
\end{proof}

\begin{construction}[Realizability for \subQuot{R}, \formalisation{https://rahulc29.github.io/realizability/Realizability.PERs.SubQuotient.html\#1866}]\label{def:realizability-for-subquotient}
  \sloppy Since the coherence condition for the pre-realizability relation $\mathsf{preReal}$ holds, we can upgrade it to an \emph{actual} realizability relation $\_\Vdash_{\subQuot{R}}\_ : \PCA \to \subQuot{R} \to \Prop{}$. We call this the \textbf{subquotient realizability relation}.
\end{construction}

\begin{lemma}\label{lemma:surjectivity-of-realizability-for-subquotient}
  For any partial equivalence class of a per, a realizer exists.
\end{lemma}
\begin{proof}
  Let $R$ be a per and $q : \subQuot{R}$ a partial equivalence class of $R$.
By the proposition-level elimination principle for set-quotients, it suffices to show existence when \Pair{\Realizer{q}}{r_{\Realizer{q}}} represents $q$.
  Since \Realizer{q} is in the domain of $R$, we have that $\Realizer{q} \approx_R \Realizer{q} \doteq \Realizer{q} \Vdash_{\subQuot{R}} [\Pair{\Realizer{q}}{r_{\Realizer{q}}}] \doteq \Realizer{q} \Vdash_{\subQuot{R}} q$ holds, and we are done.
\end{proof}

\begin{definition}[\formalisation{https://rahulc29.github.io/realizability/Realizability.PERs.SubQuotient.html\#1866}]\label{def:subquotient-assembly}
  For a per $R$, the \textbf{subquotient assembly} on $R$ has
  \begin{enumerate}
  \item The subquotient set [Definition~\ref{def:subquotient-set-per}] as carrier
  \item The subquotient realizability relation [Construction~\ref{def:realizability-for-subquotient}] as the realizability relation
  \item The proof of [Lemma~\ref{lemma:surjectivity-of-realizability-for-subquotient}] as evidence that every element is realized
  \end{enumerate}
\end{definition}

\begin{theorem}[Modesty of \subQuot{R}, \formalisation{https://rahulc29.github.io/realizability/Realizability.PERs.SubQuotient.html\#2846}]\label{theorem:subquotient-is-modest}
  For a per $R$ the subquotient assembly \subQuot{R} is modest.
\end{theorem}
\begin{proof}
  Let $r:\subQuot{R}$ and $s:\subQuot{R}$ have a shared realizer \Realizer{t}.
  By the proposition-level elimination principle for set-quotients, it suffices to show modesty for the case when we have been given representatives for $r$ and $s$.
  Let \Pair{\Realizer{r}}{r_{\Realizer{r}}} and \Pair{\Realizer{s}}{r_{\Realizer{s}}} represent $r$ and $s$ respectively. We need to show that $\Realizer{r} \approx_R \Realizer{s}$ holds.
  But this follows immediately from the transitivity of $R$ and that \Realizer{t} is a shared realizer for $r$ and $s$.
\end{proof}

\subsection{Action on Morphisms}
\label{sec:subquotient-functor-action-on-morphisms}

We have already defined the action of the subquotient functor on objects [Section~\ref{sec:subquotient-functor-action-on-objects}].
The action of the subquotient functor on morphisms is conceptually simple but verifying the well-formedness of the construction is a little messy. Uninterested readers may wish to skip the verification.

\begin{construction}[\formalisation{https://rahulc29.github.io/realizability/Realizability.PERs.SubQuotient.html\#3602}]\label{constr:subquotient-assembly-morphism-from-per-morphism}
  For pers $R$ and $S$ and a per morphism $f : R \to S$ we can construct an assembly morphism $\subQuot{f} : \subQuot{R} \to \subQuot{S}$.
\end{construction}
\begin{proof}
  First, we shall apply the set-level elimination principle for set quotients on $f$ and assume we have a representative \Realizer{f} for $f$.
  To define the underlying map of the assembly morphism, we apply the set-level elimination principle on $r : \subQuot{R}$. Let us say we have a representative \Pair{\Realizer{r}}{r_{\Realizer{r}}} for $r$. We need to choose a partial equivalence class of $S$. We choose the partial equivalence class represented by \Realizer{f \; r}. This choice is independent of the choice of representative for $r$ since \Realizer{f} is a tracker. It is also easy to see that \Realizer{f} is a tracker for the underlying map.

    It remains to see that our definition is independent of the choice of representative for $f$.
    Let \Realizer{f} and \Realizer{f'} be equivalent trackers that both represent $f$.
    Recall that, to show two assembly morphisms equal, it suffices to show that the underlying maps are pointwise equal [Lemma~\ref{lemma:assembly-morphism-extensionality}], so we need to show equality of the underlying maps for arbitrary $r : \subQuot{R}$. By the proposition-level elimination principle for set quotients on $r$, we can further assume we have a representation \Pair{\Realizer{r}}{r_{\Realizer{r}}} for $r$.

    Our goal now has been reduced to showing that the partial equivalence classes represented by \Realizer{f \; r} and \Realizer{f' \; r} are equal. This is obvious --- since \Realizer{f} and \Realizer{f'} are equivalent as trackers and $\Realizer{r} \approx_R \Realizer{r}$ holds, so we get that $\Realizer{f \; r} \approx_S \Realizer{f' \; r}$ holds and thus the partial equivalence classes represented by \Realizer{f \; r} and \Realizer{f' \; r} are equal.
\end{proof}

\begin{lemma}[\formalisation{https://rahulc29.github.io/realizability/Realizability.PERs.SubQuotient.html\#8314}]
  The mapping defined in [Construction~\ref{constr:subquotient-assembly-morphism-from-per-morphism}] is functorial, that is, it maps identities in \PER~ to corresponding identities in \Mod~ and composites in \PER~ to corresponding composites in \Mod.
\end{lemma}

\subsection{The Subquotient Functor is Fully Faithful}

\begin{lemma}[\formalisation{https://rahulc29.github.io/realizability/Realizability.PERs.SubQuotient.html\#9712}]\label{lemma:subquotient-action-morphisms-embedding}
  The action of the subquotient functor on per morphisms [Construction~\ref{constr:subquotient-assembly-morphism-from-per-morphism}] is an embedding of types, that is, it has propositional fibres.
\end{lemma}
\begin{proof}
  Let us say we have been given a morphism $f : \subQuot{R} \to \subQuot{S}$ along with two per morphisms $x$ and $y$ in the fibre space over $f$, so we have that both $\subQuot{x} = f$ and $\subQuot{y} = f$ hold. Since assembly morphisms form a set, showing we have propositional fibres reduces to showing that $x = y$ holds. By the proposition-level elimination principle for set-quotients, it suffices to show that $x = y$ holds when we have actual representatives for $x$ and $y$. Let \Realizer{x} and \Realizer{y} be trackers [Definition~\ref{def:tracker}] that represent $x$ and $y$ respectively.

  We can refine our goal to showing that $[ \Realizer{x} ] = [ \Realizer{y} ]$ holds, for which it suffices to show that \Realizer{x} and \Realizer{y} are equivalent as trackers [Definition~\ref{def:equivalence-of-trackers}]. To show that \Realizer{x} and \Realizer{y} are equivalent as trackers, we need to show that for any \Realizer{r} such that $\Realizer{r} \approx_R \Realizer{r}$ holds, we have that $\Realizer{x \; r} \approx_S \Realizer{y \; r}$ holds.

  \sloppy Let us zoom out and think about what meaning we can make of \Realizer{x \; r} and \Realizer{y \; r}. We know that $\Realizer{r} \approx_R \Realizer{r}$ holds; witnessed by, say $e_{\Realizer{r}} : \Realizer{r} \approx_R \Realizer{r}$. This means we have a term $\Pair{\Realizer{r}}{e_{\Realizer{r}}}$ of type $\Dom{R}$. We can apply $f$ to it to get $f([\Pair{\Realizer{r}}{e_{\Realizer{r}}}]) : \subQuot{S}$.
  Since $\subQuot{x} = f$ and $\subQuot{y} = f$ hold, we have that $\subQuot{x}([\Pair{\Realizer{r}}{e_{\Realizer{r}}}]) = \subQuot{y}([\Pair{\Realizer{r}}{e_{\Realizer{r}}}])$ holds.
  Moreover, since \Realizer{x} and \Realizer{y} and are trackers for $x$ and $y$ respectively, we have that $\Realizer{x \; r} \Vdash_{\subQuot{S}} \subQuot{x}([\Pair{\Realizer{r}}{e_{\Realizer{r}}}])$ and $\Realizer{y \; r} \Vdash_{\subQuot{S}} \subQuot{y}([\Pair{\Realizer{r}}{e_{\Realizer{r}}}])$ hold. Another way to write this is to say that $\Realizer{x \; r} \Vdash_{\subQuot{S}} f([\Pair{\Realizer{r}}{e_{\Realizer{r}}}])$ and $\Realizer{y \; r} \Vdash_{\subQuot{S}} f([\Pair{\Realizer{r}}{e_{\Realizer{r}}}])$ hold.

  \sloppy To go beyond $\Realizer{x \; r} \Vdash_{\subQuot{S}} f([\Pair{\Realizer{r}}{e_{\Realizer{r}}}])$ and $\Realizer{y \; r} \Vdash_{\subQuot{S}} f([\Pair{\Realizer{r}}{e_{\Realizer{r}}}])$, we will apply the proposition-level elimination principle for set-quotients; this time on $f([\Pair{\Realizer{r}}{e_{\Realizer{r}}}])$. Let us say that $\Pair{\Realizer{s}}{e_{\Realizer{s}}}$ represents $f([\Pair{\Realizer{r}}{e_{\Realizer{r}}}])$. In this case, $\Realizer{x \; r} \Vdash_{\subQuot{S}} f([\Pair{\Realizer{r}}{e_{\Realizer{r}}}])$ unfolds to $\Realizer{x \; r}  \Vdash_{\subQuot{S}} [\Pair{\Realizer{s}}{e_{\Realizer{s}}}]$, which further unfolds to $\Realizer{x \; r} \approx_S \Realizer{s}$. Likewise, $\Realizer{y \; r} \Vdash_{\subQuot{S}} f([\Pair{\Realizer{r}}{e_{\Realizer{r}}}])$ unfolds to $\Realizer{y \; r} \approx_S \Realizer{s}$. Now, since $\Realizer{x \; r} \approx_S \Realizer{s}$ and $\Realizer{y \; r} \approx_S \Realizer{s}$ hold, it is indeed the case that $\Realizer{x \; r} \approx_S \Realizer{y \; r}$ holds, and we are done!
\end{proof}
To show that an embedding can be upgraded to an equivalence we need to show that every fibre space is inhabited. However, in our case, we can go one step further, we can find a choice of inhabitant of every fibre space.

We will proceed in four steps. First, we will show that any tracker of an assembly morphism $f : \subQuot{R} \to \subQuot{S}$ also tracks a per morphism $R \to S$. Secondly, we will show that the actual choice of tracker is immaterial, that is, if \Realizer{a} and \Realizer{b} both track $f$ then they necessarily track the \emph{same} per morphism. In the third step, we will define a function from assembly morphisms $\subQuot{R} \to \subQuot{S}$ to per morphisms $R \to S$. Finally, we will check that the per morphism we obtain is inhabits the fibre space over $f$.

\begin{lemma}[\formalisation{https://rahulc29.github.io/realizability/Realizability.PERs.SubQuotient.html\#6520}]\label{lemma:morphism-tracker-per-tracker}
  A realizer $\Realizer{a} : \PCA$ that tracks a morphism of subquotient assemblies $f : \subQuot{R} \to \subQuot{S}$ also tracks a per morphism $R \to S$.
\end{lemma}
\begin{proof}
  Let \Realizer{a} be a tracker for $f : \subQuot{R} \to \subQuot{S}$. We need to show that for any \Realizer{r} and \Realizer{r'} that are related by $R$, \Realizer{a \; r} and \Realizer{a \; r'} are related by $S$. Notice that both \Realizer{r} and \Realizer{r'} are related to themselves so they are in the domain of $R$. Let $r_{\Realizer{r}} : \Realizer{r} \approx_R \Realizer{r}$ and $r_{\Realizer{r'}} : \Realizer{r'} \approx_R \Realizer{r'}$ respectively witness \Realizer{r} and \Realizer{r'} being related to themselves. Observe that \Realizer{a \; r} and \Realizer{a \; r'} are necessarily related to $f([\Pair{\Realizer{r}}{r_{\Realizer{r}}}])$ and therefore must be related to each other.

  To verify the details, apply the proposition-level elimination rule on $f([\Pair{\Realizer{r}}{r_{\Realizer{r}}}])$. We assume that $f([\Pair{\Realizer{r}}{r_{\Realizer{r}}}])$ is represented by $\Pair{\Realizer{s}}{r_{\Realizer{s}}} : \Dom{S}$.

  We calculate to see that 
  \begin{align*}
    &\Realizer{r} \Vdash_{\subQuot{R}} [\Pair{\Realizer{r}}{r_{\Realizer{r}}}] \\
    &\reasoning{\Realizer{a}\text{ tracks }f} \\
    &\Realizer{a \; r} \Vdash_{\subQuot{S}} f([\Pair{\Realizer{r}}{r_{\Realizer{r}}}]) \\
    &\doteq \reasoning{f([\Pair{\Realizer{r}}{r_{\Realizer{r}}}])\text{ is represented by }\Pair{\Realizer{s}}{r_{\Realizer{s}}}}\\
    &\Realizer{a \; r} \Vdash_{\subQuot{S}} [\Pair{\Realizer{s}}{r_{\Realizer{s}}}] \\
    &\doteq \reasoning{\text{by definition of subquotient realizability}} \\
    &\Realizer{a \; r} \approx_S \Realizer{s}
  \end{align*}
  holds.
  
  Similarly, we have that
  \begin{align*}
    &\Realizer{r'} \Vdash_{\subQuot{R}} [\Pair{\Realizer{r'}}{r_{\Realizer{r'}}}] \\
    &\reasoning{\Realizer{a}\text{ tracks }f} \\
    &\Realizer{a \; r'} \Vdash_{\subQuot{S}} f([\Pair{\Realizer{r'}}{r_{\Realizer{r'}}}]) \\
    &= \reasoning{[\Pair{\Realizer{r'}}{r_{\Realizer{r}}}] = [\Pair{\Realizer{r}}{r_{\Realizer{r}}}]} \\
    &\Realizer{a \; r'} \Vdash_{\subQuot{S}} f([\Pair{\Realizer{r}}{r_{\Realizer{r}}}]) \\
    &\doteq \reasoning{f([\Pair{\Realizer{r}}{r_{\Realizer{r}}}])\text{ is represented by }\Pair{\Realizer{s}}{r_{\Realizer{s}}}}\\
    &\Realizer{a \; r'} \Vdash_{\subQuot{S}} [\Pair{\Realizer{s}}{r_{\Realizer{s}}}] \\
    &\doteq \reasoning{\text{by definition of subquotient realizability}} \\
    &\Realizer{a \; r'} \approx_S \Realizer{s}
  \end{align*}
  holds. By symmetry and transitivity of $S$, $\Realizer{a \; r} \approx_S \Realizer{a \; r'}$ holds. This concludes the proof.
\end{proof}

\begin{lemma}[\formalisation{https://rahulc29.github.io/realizability/Realizability.PERs.SubQuotient.html\#7563}]\label{lemma:tracker-2-constant}
  If realizers $\Realizer{a}, \Realizer{b} : \PCA$ track a morphism of subquotient assemblies $f : \subQuot{R} \to \subQuot{S}$ then they track the same per morphism $R \to S$.
\end{lemma}
\begin{proof}
  Let \Realizer{a} and \Realizer{b} be trackers for $f : \subQuot{R} \to \subQuot{S}$. To show that they track the same per morphism, it suffices to show that they are equivalent as trackers.
  
  We need to show that for any $\Pair{\Realizer{r}}{r_{\Realizer{r}}} : \Dom{R}$ we have that $\Realizer{a \; r} \approx_S \Realizer{b \; r}$ holds.

  This proof mimics the proof of the previous lemma [Lemma~\ref{lemma:morphism-tracker-per-tracker}].

  We apply the proposition-level elimination rule for set-quotients on $f([\Pair{\Realizer{r}}{r_{\Realizer{r}}}])$. We assume $\Pair{\Realizer{s}}{r_{\Realizer{s}}} : \Dom{S}$ represents $f([\Pair{\Realizer{r}}{r_{\Realizer{r}}}])$.

  Since both \Realizer{a} and \Realizer{b} are trackers for $f$, we have that $\Realizer{a \; r} \approx_S \Realizer{s}$ and $\Realizer{b \; r} \approx_S \Realizer{s}$ hold, so that $\Realizer{a \; r} \approx_S \Realizer{b \; r}$ holds as well.
\end{proof}

\begin{construction}
  For any morphism of subquotient assemblies $f : \subQuot{R} \to \subQuot{S}$ we can find a per morphism $\perify{f} : R \to S$ such that any tracker of $f$ is also a tracker of \perify{f}. 
\end{construction}
\begin{proof}
  We wish to map a tracker \Realizer{a} for $f$ to the per morphism tracked by \Realizer{a} [Lemma~\ref{lemma:morphism-tracker-per-tracker}]. Since the actual choice of tracker is immaterial [Lemma~\ref{lemma:tracker-2-constant}], we can apply the set-level recursion rule for propositional truncation on the witness of $f$ being tracked, giving rise to the map $\perify{f} : R \to S$.
\end{proof}

\begin{lemma}[\formalisation{https://rahulc29.github.io/realizability/Realizability.PERs.SubQuotient.html\#10985}]\label{lemma:canonical-inhabitant-fibre-space-subquotient}
  For any $f : \subQuot{R} \to \subQuot{S}$ we have that \perify{f} is an inhabitant of the fibre space over $f$, that is, $\subQuot{\perify{f}} = f$ holds.
\end{lemma}
\begin{proof}
  This proof is simple, but we need to get our hands dirty with a lot of elimination rules!

  By assembly morphism extensionality [Lemma~\ref{lemma:assembly-morphism-extensionality}], it suffices to show that the underlying maps are pointwise equal.

  Let \Pair{\Realizer{r}}{r_{\Realizer{r}}} be an arbitrary term of type \Dom{R}. We will show that $\subQuot{\perify{f}}([\Pair{\Realizer{r}}{r_{\Realizer{r}}}]) = f([\Pair{\Realizer{r}}{r_{\Realizer{r}}}])$ holds. By the proposition-level elimination principle for set-quotients, this is sufficient.

  We apply the elimination rule for propositional truncations on the witness of $f$ being tracked. Assume that \Realizer{t} is a tracker for $f$, so that $\perify{f}$ is the per morphism tracked by \Realizer{t}.

  We calculate
  \begin{align*}
    &\subQuot{\perify{f}}([\Pair{\Realizer{r}}{r_{\Realizer{r}}}]) \\
    &\doteq \reasoning{\Realizer{t}\text{ tracks }f} \\
    &\subQuot{[ \Realizer{t} ]}([\Pair{\Realizer{r}}{r_{\Realizer{r}}}]) \\
    &\doteq \reasoning{\text{by the definition of }\subQuot{\_}\text{[Construction~\ref{constr:subquotient-assembly-morphism-from-per-morphism}]}} \\
    &[ \Realizer{t \; r} ]
  \end{align*}
  to reduce our goal to showing that $[ \Realizer{t \; r} ] = f([\Pair{\Realizer{r}}{r_{\Realizer{r}}}]$ holds.

  To get rid of this final obstacle, we again apply proposition-level set-quotient elimination on $f([\Pair{\Realizer{r}}{r_{\Realizer{r}}}])$. We assume that $\Pair{\Realizer{s}}{r_{\Realizer{s}}} : \Dom{S}$ represents $f([\Pair{\Realizer{r}}{r_{\Realizer{r}}}]$.

  Hopefully, the reader has recognised the pattern. Since \Realizer{t} tracks $f$ we have that $\Realizer{t \; r} \Vdash_{\subQuot{S}} f([\Pair{\Realizer{r}}{r_{\Realizer{r}}}]) \doteq \Realizer{t \; r} \Vdash_{\subQuot{S}} [\Pair{\Realizer{s}}{r_{\Realizer{s}}}] \doteq \Realizer{t \; r} \approx_S \Realizer{s}$ holds. This means that the partial equivalence classes $[\Realizer{t \; r}]$ and $[\Realizer{s}]$ are equal, and our proof is complete.
\end{proof}

Since the action of the subquotient functor on per morphisms is an embedding of types [Lemma~\ref{lemma:subquotient-action-morphisms-embedding}], and every fibre space is inhabited [Lemma~\ref{lemma:canonical-inhabitant-fibre-space-subquotient}]; we have shown:

\begin{theorem}[\formalisation{https://rahulc29.github.io/realizability/Realizability.PERs.SubQuotient.html\#12415}]\label{theorem:subquotient-fully-faithful}
  The subquotient functor is \emph{fully faithful}, that is, the action of the subquotient functor on per morphisms is an equivalence of types.
\end{theorem}

\section{The Subquotient Functor is an Equivalence of Categories}
\label{sec:modest-sets-equivalent-to-pers}

\sloppy Our strategy for building the equivalence is to show that the subquotient functor $\subQuot{\_} : \PER \to \Mod$ we just defined in the previous section [Section~\ref{sec:subquotient-functor}] is a \textbf{split essentially surjective functor}. To do this, we must give, for any modest set $M$, a per $R$ such that $M$ is isomorphic to $\subQuot{R}$. We call this $R$ the \textbf{canonical per} of the modest set $M$ and denote it $\canonicalPER{M}$. Much of this section is devoted to constructing the forward $M \to \subQuot{\canonicalPER{M}}$ and backward $\subQuot{\canonicalPER{M}} \to M$ directions of the isomorphism $M \cong \subQuot{\canonicalPER{M}}$.

\subsection{Canonical PER of a Modest Set}
\label{sec:canonical-per-of-a-modest-set}

Before we define the canonical per, we prove a helpful lemma.

\begin{lemma}\label{lemma:canonical-per-type-prop}
  For any modest set $M$, and realizers $\Realizer{a} : \PCA$ and $\Realizer{b} : \PCA$, the type $(x : M) \times \Realizer{a} \Vdash_M x \times \Realizer{b} \Vdash_M x$ is a proposition.
\end{lemma}
\begin{proof}
  Since $\_ \Vdash_M \_$ is a proposition, it suffices to show that any $x : M$ and $x' : M$ that are realized by both \Realizer{a} and \Realizer{b} are necessarily equal. But this follows immediately from the modesty of $M$, since $x$ and $x'$ share not one but two realizers!
\end{proof}

\begin{definition}[\formalisation{https://rahulc29.github.io/realizability/Realizability.Modest.CanonicalPER.html\#1890}]
  The canonical per of a modest set $M$ relates \Realizer{a} and \Realizer{b} if $(x : M) \times \Realizer{a} \Vdash_M x \times \Realizer{b} \Vdash_M x$ holds.
\end{definition}
\begin{proof}
  The type is a proposition [Lemma~\ref{lemma:canonical-per-type-prop}]. It remains to show symmetry and transitivity.

  Symmetry obviously holds by a simple re-arrangement of the data.

  For transitivity, consider \Realizer{a}, \Realizer{b} and \Realizer{c} elements of \PCA~ such that we have $x : M$ realized by \Realizer{a} and \Realizer{b}, and $x' : M$ realized by \Realizer{b} and \Realizer{c}. Note that $x = x'$ holds since $x$ and $x'$ share a realizer \Realizer{b}. As such, we have that $\Realizer{a} \Vdash x'$ and $\Realizer{c} \Vdash x$ hold as well.

  We need to choose a $x'' : M$ that is realized by both \Realizer{a} and \Realizer{c}. Obviously, we can choose $x'' := x$ for this.
\end{proof}

\subsection{The Forward Direction $M \to \subQuot{\canonicalPER{M}}$}

To define the forward direction, we would need to introduce an element of \subQuot{\canonicalPER{M}},
which would require us to furnish an element $x : M$ along with a realizer \Realizer{a} that realizes $x$.
Of course, we have an input $M$, so the $x : M$ is taken care of, but how do we possibly find an actual realizer for arbitrary $x : M$? We do not!
The crucial observation is that the choice of realizer does \emph{not} matter --- the main map we will be defining will turn out to be na\"ively constant. This will allow us to appeal to the recursion principle of propositional truncation to define our underlying map. Showing that this map is tracked is then straightforward.

\begin{definition}[\formalisation{https://rahulc29.github.io/realizability/Realizability.Modest.SubQuotientCanonicalPERIso.html\#4184}]\label{def:main-map-forward}
  For any $x : M$, we define a function
  \begin{align*}
    &\fwdMain_x : (\Realizer{a} : \PCA) \times \Realizer{a} \Vdash x \to \subQuot{\canonicalPER{M}} \\
    &\fwdMain_x \Define \lambda \Pair{\Realizer{a}}{e_{\Realizer{a}}[x]} \to [ \langle \Realizer{a} , \langle x , e_{\Realizer{a}}[x] , e_{\Realizer{a}}[x]  \rangle \rangle ]
  \end{align*}
  mapping any \Realizer{a} that realizes $x$ to the partial equivalence class determined by \Realizer{a}.
\end{definition}

\begin{lemma}[\formalisation{https://rahulc29.github.io/realizability/Realizability.Modest.SubQuotientCanonicalPERIso.html\#4303}]
  For any $x : M$, the map $\fwdMain_x \;$ is na\"ively constant.
\end{lemma}
\begin{proof}
  We need to show that for any $\Pair{\Realizer{r}}{e_{\Realizer{r}}[x]}$ and $\Pair{\Realizer{s}}{e_{\Realizer{s}}[x]}$ in $(\Realizer{a} : \PCA) \times \Realizer{a} \Vdash x$ we have that the partial equivalence classes $[ \langle \Realizer{r} , \langle x , e_{\Realizer{r}}[x] , e_{\Realizer{r}}[x] \rangle \rangle ]$ and $[ \langle \Realizer{s}, \langle x , e_{\Realizer{s}}[x], e_{\Realizer{s}}[x] \rangle \rangle ]$ are equal, for which it suffices to show that $\Realizer{r} \approx_{\canonicalPER{M}} \Realizer{s}$ holds. We need to furnish an element of $M$ that is realized by both \Realizer{r} and \Realizer{s}. By assumption, $x$ is a valid choice.
\end{proof}

\begin{construction}[\formalisation{https://rahulc29.github.io/realizability/Realizability.Modest.SubQuotientCanonicalPERIso.html\#4069}]\label{constr:fwd-underlying}
  We can construct the map $\fwd : M \to \subQuot{\canonicalPER{M}}$.
\end{construction}
\begin{proof}
  We need to map a given $x : M$ to a partial equivalence class $\mathsf{SubQuotient}\\ (\canonicalPER{M})$.

  Since $M$ is an assembly, we have a witness $e_x : \exists (\Realizer{a} : \PCA) \; \Realizer{a} \Vdash_M x$ that a realizer for $x$ exists. Since $\fwdMain_x$ is a na\"ively constant map, we can apply the set-level recursion principle for propositional truncation on $e_x$ and $\fwdMain_x$.

  This gives us an actual partial equivalence class \subQuot{\canonicalPER{M}} and thus our construction is complete.
\end{proof}

\sloppy Notably, when $e : \exists (\Realizer{a} : \PCA) \; \Realizer{a} \Vdash_M x$ is of the form $\ptInc{\Pair{\Realizer{a}}{e_{\Realizer{a}}[x]}}$, we have that $\fwd(x) \doteq \; \fwdMain_x (\Pair{\Realizer{a}}{e_{\Realizer{a}}[x]}) \doteq \; [ \langle \Realizer{a} , \langle x , e_{\Realizer{a}}[x] , e_{\Realizer{a}}[x]  \rangle \rangle ]$ holds by definition.

\begin{lemma}[\formalisation{https://rahulc29.github.io/realizability/Realizability.Modest.SubQuotientCanonicalPERIso.html\#4069}]\label{lemma:fwd-tracked}
  The map $\fwd \;$is tracked.
\end{lemma}
\begin{proof}
  We need to show that there exists some $\Realizer{t} : \PCA$ such that for any $x : M$ and for any \Realizer{a} that realizes $x$ we have that \Realizer{t \; a} realizes $\fwd(x)$.
  
  We claim that the identity combinator \Realizer{i} is a tracker for \fwd.

  Let us assume we have been given $x : M$ and a realizer \Realizer{a} for $x$. We need to show that $\Realizer{i \; a} = \Realizer{a}$ realizes $\fwd(x)$.

  Since $M$ is an assembly, we have evidence $e_x$ that $x$ has a realizer. Recall that $\fwd(x)$ is defined by recursion on this very $e_x$. By the proposition-level elimination principle for propositional truncations, it is enough to show that \Realizer{a} realizes $\fwd(x)$ when $e_x$ is of the form \ptInc{\Pair{\Realizer{b}}{e_{\Realizer{b}}[x]}} where $\Realizer{b} : \PCA$ is a realizer and $e_{\Realizer{b}}[x]$ is evidence that $b \Vdash x$ holds.

  By calculating
  \begin{align*}
    &\Realizer{a} \Vdash_{\subQuot{\canonicalPER{M}}} \fwd(x) \\
    &\doteq \reasoning{e_x\text{ is of the form }\ptInc{\Pair{\Realizer{b}}{e_{\Realizer{b}}[x]}}} \\
    &\Realizer{a} \Vdash_{\subQuot{\canonicalPER{M}}} \fwdMain_x(\Pair{b}{e_{\Realizer{b}}[x]}) \\
    &\doteq \reasoning{\text{by definition of }\fwdMain} \\
    &\Realizer{a} \Vdash_{\subQuot{\canonicalPER{M}}} [ \langle \Realizer{b} , \langle x , e_{\Realizer{b}}[x] , e_{\Realizer{b}}[x]  \rangle \rangle ] \\
    &\doteq \reasoning{\text{by definition of subquotient realizability [Construction~\ref{def:realizability-for-subquotient}]}} \\
    &\Realizer{a} \approx_{\canonicalPER{M}} \Realizer{b}
  \end{align*}
  we can refine our goal to showing that $\Realizer{a} \approx_{\canonicalPER{M}} \Realizer{b}$ holds. This is trivial, since $x$ is realized by both \Realizer{a} and \Realizer{b}.
\end{proof}

Thus, we have shown :

\begin{corollary}[\formalisation{https://rahulc29.github.io/realizability/Realizability.Modest.SubQuotientCanonicalPERIso.html\#4069}]
  \sloppy We have the assembly morphism $\fwd : M \to \subQuot {\canonicalPER{M}}$ with underlying map $\fwd \;$[Construction~\ref{constr:fwd-underlying}] and the proof of [Lemma~\ref{lemma:fwd-tracked}] as evidence of $\fwd \;$being tracked.
\end{corollary}

\subsection{The Backward Direction $\subQuot{\canonicalPER{M}} \to M$}
\label{sec:backwards-direction}

\sloppy We shall define a function $\bwdMain : \Dom{\canonicalPER{M}} \to M$ and then show that the appropriate coherence conditions hold for $\bwdMain$, so that we can apply the set-level recursion principle for set quotients to get a function $\bwd : \subQuot{\canonicalPER{M}} \to M$. After that, it is easy to show that $\bwd \;$is tracked and thus we will get a assembly morphism from \subQuot{\canonicalPER{M}} to $M$.

\begin{definition}[\formalisation{https://rahulc29.github.io/realizability/Realizability.Modest.SubQuotientCanonicalPERIso.html\#2716}]
  The function $\bwdMain : \Dom{\canonicalPER{M}} \to M$ is defined by
  \begin{align*}
    &\bwdMain : \Dom{\canonicalPER{M}} \to M \\
    &\bwdMain \Define \lambda \Pair{\Realizer{a}}{\langle x, e_{\Realizer{a}}[x], e'_{\Realizer{a}}[x] \rangle} \to x
  \end{align*}
\end{definition}

\begin{lemma}[\formalisation{https://rahulc29.github.io/realizability/Realizability.Modest.SubQuotientCanonicalPERIso.html\#2813}]\label{lemma:bwdMain}
  The function $\bwdMain : \Dom{\canonicalPER{M}} \to M$ maps elements of \Dom{\canonicalPER{M}} that are related by the \canonicalPER{M} relation to equal elements of $M$.
\end{lemma}
\begin{proof}
  Let us say we have $\Pair{\Realizer{a}}{\langle x, e_{\Realizer{a}}[x], e'_{\Realizer{a}}[x] \rangle}$ and $\Pair{\Realizer{b}}{\langle x', e_{\Realizer{b}}[x'], e'_{\Realizer{b}}[x'] \rangle}$ in $\Dom{\canonicalPER{M}}$ along with evidence $\langle x'', e_{\Realizer{a}}[x''], e_{\Realizer{b}}[x''] \rangle$ that they are related by the \canonicalPER{M} relation. We need to show that $\bwdMain(\Pair{\Realizer{a}}{\langle x, e_{\Realizer{a}}[x], e'_{\Realizer{a}}[x] \rangle}) \doteq x$ equals $\bwdMain(\Pair{\Realizer{b}}{\langle x', e_{\Realizer{b}}[x'], e'_{\Realizer{b}}[x'] \rangle}) \doteq x'$.

  A simple calculation
  \begin{align*}
    &x \\
    &=\reasoning{x\text{ and }x''\text{ share the realizer }\Realizer{a}} \\
    &x'' \\
    &=\reasoning{x''\text{ and }x'\text{ share the realizer }\Realizer{b}} \\
    &x'
  \end{align*}
  suffices to show this.
\end{proof}

\begin{construction}[\formalisation{https://rahulc29.github.io/realizability/Realizability.Modest.SubQuotientCanonicalPERIso.html\#2560}]
  We have a map obtained by applying set-quotient recursion to upgrade the map $\bwdMain$ to a full-fledged map $\bwd : \subQuot{\canonicalPER{M}} \to M$.
\end{construction}

All that remains to check is that \bwd~ is tracked.

\begin{lemma}[\formalisation{https://rahulc29.github.io/realizability/Realizability.Modest.SubQuotientCanonicalPERIso.html\#2560}]\label{lemma:bwd-realizability}
  If \Realizer{a} : \PCA~ realizes $q : \subQuot{\canonicalPER{M}}$, then it also realizes $\bwd(q)$.
\end{lemma}
\begin{proof}
  By the proposition-level set-quotient elimination principle on $q$, it is enough to show the lemma when we have a representative for $q$. Let $q$ be represented by \Pair{\Realizer{q}}{\Triple{x}{e_{\Realizer{q}}[x]}{e'_{\Realizer{q}}}}. We need to show that if \Realizer{a} : \PCA~ realizes $[\Pair{\Realizer{q}}{\Triple{x}{e_{\Realizer{q}}[x]}{e'_{\Realizer{q}}[x]}}] : \subQuot{\canonicalPER{M}}$ then it also realizes $\bwd([\Pair{\Realizer{q}}{\Triple{x}{e_{\Realizer{q}}[x]}{e'_{\Realizer{q}}[x]}}]) \doteq \bwdMain(\Pair{\Realizer{q}}{\Triple{x}{e_{\Realizer{q}}[x]}{e'_{\Realizer{q}}}}) \doteq x$.

  Since \Realizer{a} realizes $[\Pair{\Realizer{q}}{\Triple{x}{e_{\Realizer{q}}[x]}{e'_{\Realizer{q}}[x]}}]$, we have that $\Realizer{a} \approx_{\canonicalPER{M}} \Realizer{q}$ holds, which in turn means that we have an $x' : M$ realized by both \Realizer{a} and \Realizer{q}. Now, $x$ and $x'$ share a realizer in \Realizer{q}, so they are equal. Because \Realizer{a} already realizes $x'$ and $x'$ equals $x$, \Realizer{a} also realizes $x$. This concludes the proof.
\end{proof}

\begin{lemma}[\formalisation{https://rahulc29.github.io/realizability/Realizability.Modest.SubQuotientCanonicalPERIso.html\#2560}]
  The map $\bwd : \subQuot{\canonicalPER{M}} \to M$ is tracked and is thus an assembly morphism.
\end{lemma}
\begin{proof}
  We claim that \Realizer{i} is a tracker, that is, for any $q : \subQuot{\canonicalPER{M}}$ and \Realizer{a} that realizes $q$ we have that \Realizer{i \; a} realizes $\bwd(q)$ holds. Because \Realizer{i \; a} is equal to \Realizer{a}, our goal is reduced to showing that \Realizer{a} realizes $\bwd(q)$; at which point we can simply apply the previous lemma [Lemma~\ref{lemma:bwd-realizability}].
\end{proof}

\subsection{Isomorphism Coherences}

Our final goal is to prove that $\bwd \; \circ \; \fwd = \id_{M}$ and $\fwd \circ \bwd = \id_{\subQuot{\canonicalPER{M}}}$ hold.

Since assembly morphisms form a set, both of these types are propositions. As such, we can apply the proposition-level elimination principles for set-quotients and propositional truncations to establish these coherences.

\begin{lemma}[\formalisation{https://rahulc29.github.io/realizability/Realizability.Modest.SubQuotientCanonicalPERIso.html\#6403}]\label{lemma:fwd-bwd-equals-id}
  We have that $\fwd \circ \bwd = \id_{\subQuot{\canonicalPER{M}}}$ holds.
\end{lemma}
\begin{proof}
  By assembly morphism extensionality [Lemma~\ref{lemma:assembly-morphism-extensionality}] and proposition-level elimination rule for set quotients, it suffices to show that for any $\Pair{\Realizer{q}}{\Triple{x}{e_{\Realizer{q}}[x]}{e'_{\Realizer{q}}[x]}} : \Dom{\canonicalPER{M}}$, $\fwd (\bwd([ \Pair{\Realizer{q}}{\Triple{x}{e_{\Realizer{q}}[x]}{e'_{\Realizer{q}}[x]}} ]))$ equals $[ \Pair{\Realizer{q}}{\Triple{x}{e_{\Realizer{q}}[x]}{e'_{\Realizer{q}}[x]}} ]$.

Recall that \fwd~ is defined by recursion on the witness $e_x : \exists (\Realizer{a} : \PCA) \; \Realizer{a} \Vdash_M x$ that there exists a realizer for $x$. By the elimination rule for propositional truncations, we can assume that $e_x$ is of the form $\ptInc{\Pair{\Realizer{p}}{e_{\Realizer{p}}[x]}}$, where \Realizer{p} is a realizer and $e_{\Realizer{p}}[x]$ witnesses that $\Realizer{p} \Vdash_M x$ holds.

  We calculate
  \begin{align*}
    &\fwd (\bwd([ \Pair{\Realizer{q}}{\Triple{x}{e_{\Realizer{q}}[x]}{e'_{\Realizer{q}}[x]}} ])) \\ 
    &\doteq \reasoning{\text{by definition of }\bwd} \\
    &\fwd (\bwdMain(\Pair{\Realizer{q}}{\Triple{x}{e_{\Realizer{q}}[x]}{e'_{\Realizer{q}}[x]}}) \\
    &\doteq \reasoning{\text{by definition of }\bwdMain} \\
    &\fwd (x) \\
    &\doteq \reasoning{e_x\text{ is of the form }\ptInc{\Pair{\Realizer{p}}{e_{\Realizer{p}}[x]}}} \\
    &\fwdMain (\Pair{\Realizer{p}}{e_{\Realizer{p}}[x]}) \\
    &\doteq \reasoning{\text{by definition of }\fwdMain} \\
    &[ \Pair{\Realizer{p}}{\langle x ,  e_{\Realizer{p}}[x],  e_{\Realizer{p}}[x] \rangle} ]
  \end{align*}
  to refine our goal to showing that $[ \Pair{\Realizer{p}}{\langle x ,  e_{\Realizer{p}}[x],  e_{\Realizer{p}}[x] \rangle} ] = [ \Pair{\Realizer{q}}{\langle x , e_{\Realizer{q}}[x] , e'_{\Realizer{q}}[x] \rangle} ]$ holds. This amounts to showing that
  $\Realizer{p} \approx_{\canonicalPER{M}} \Realizer{q}$ holds, which is obviously true since $x$ is realized by both \Realizer{p} and \Realizer{q}.
\end{proof}

\begin{lemma}[\formalisation{https://rahulc29.github.io/realizability/Realizability.Modest.SubQuotientCanonicalPERIso.html\#5855}]\label{lemma:bwd-fwd-equals-id}
  We have that $\bwd \circ \fwd = \id_{M}$ holds.
\end{lemma}
\begin{proof}
  Similar to the previous lemma [Lemma~\ref{lemma:fwd-bwd-equals-id}], it suffices to show pointwise equality for the underlying functions on both sides. Our goal, then, is to show that for any $x : M$, $\bwd (\fwd(x)) = x$ holds.

  By the proposition-level elimination principle of the propositional truncation, we can assume that the witness $e_{x} : \exists (\Realizer{a} : \PCA) \; \Realizer{a} \Vdash_M x$ is of the form $\ptInc{\Pair{\Realizer{x}}{e_{\Realizer{x}}[x]}}$.

  We calculate
  \begin{align*}
    &\bwd (\fwd(x)) \\
    &\doteq \reasoning{\text{by definition of }\fwd} \\
    &\bwd (\fwdMain(\Pair{\Realizer{x}}{e_{\Realizer{x}}[x]}) \\
    &\doteq \reasoning{\text{by definition of }\fwdMain} \\
    &\bwd ([\Pair{\Realizer{x}}{\langle x, e_{\Realizer{x}}[x], e_{\Realizer{x}}[x] \rangle}]) \\
    &\doteq \reasoning{\text{by definition of }\bwd} \\
    &\bwdMain (\Pair{\Realizer{x}}{\langle x, e_{\Realizer{x}}[x], e_{\Realizer{x}}[x] \rangle}) \\
    &\doteq \reasoning{\text{by definition of }\bwdMain} \\
    &x
  \end{align*}
  to see that the intended equality holds by definition!
\end{proof}

The coherences [Lemma~\ref{lemma:fwd-bwd-equals-id}] and [Lemma~\ref{lemma:bwd-fwd-equals-id}] together imply :
\begin{theorem}[\formalisation{https://rahulc29.github.io/realizability/Realizability.Modest.SubQuotientCanonicalPERIso.html\#5008}]\label{theorem:subquotient-split-essentially-surjective}
  For any modest set $M$, we have an isomorphism $M \cong \\ \subQuot{\canonicalPER{M}}$.
  In other words, the subquotient functor $\subQuot{\_} : \PER \to \Mod$ is a \textbf{split essentially surjective} functor.
\end{theorem}

\section{Concluding Remarks}
\label{sec:conclusion}

Our hike is over, and it is time to enjoy the view. At long last, we can state and prove the theorem we had originally promised!

\begin{theorem}
  The subquotient functor is an \emph{equivalence} of categories \PER~ and \Mod.
\end{theorem}
\begin{proof}
  The subquotient functor is both fully faithful [Theorem~\ref{theorem:subquotient-fully-faithful}] and split essentially surjective [Theorem~\ref{theorem:subquotient-split-essentially-surjective}].
\end{proof}

\clearpage
\printbibliography

\end{document}